\documentclass[11pt,reqno]{amsart}
\usepackage{amssymb,mathrsfs,graphicx,subfigure, enumerate}
\usepackage{amsmath,amsfonts,amssymb,amscd,amsthm,bbm}
\usepackage{graphicx,colortbl,here}
\usepackage{extpfeil}
\usepackage{color}
\numberwithin{figure}{section}

\usepackage[utf8]{inputenc}
\usepackage[T1]{fontenc}

\def \rpsi_i {|\psi_i \rangle}
\def \lpsi_i {\langle \psi_i|}
\def \lrpsi_i{\langle \psi_i | \psi_i \rangle}
\def \rpsi_k {|\psi_k \rangle}
\def \lpsi_k {\langle \psi_k|}
\def \lrpsi_k{\langle \psi_k | \psi_k \rangle}

\newcommand{\opnorm}[1]{{ \vert\kern-0.25ex \vert\kern-0.25ex \vert #1 
     \vert\kern-0.25ex \vert\kern-0.25ex \vert}}

\newcommand{\bbr}{\mathbb R}

\newcommand{\kp}{\kappa}

\newcommand{\veps}{\varepsilon}

\newcommand{\ba}{\begin{aligned}}
\newcommand{\ea}{\end{aligned}}

\newcommand{\be}{\begin{equation}}
\newcommand{\ee}{\end{equation}}

\newcommand{\dd}{  {\textup{d}} }
\newcommand{\dt} {    {  \textup{d}t   }    }
\newcommand{\tr}{{\textup{tr}}}

\newcommand{\St}{  {\textup{St}}(p,n) }
\newcommand{\init}{  {\textup{in}} }

\newtheorem{theorem}{Theorem}[section]
\newtheorem{lemma}{Lemma}[section]

\newtheorem{remark}{Remark}[section]
\newtheorem{definition}{Definition}[section]

\begin{document}

\title[ ]{Asymptotic stability of the high-dimensional Kuramoto model on Stiefel manifolds}
%\title[ ]{Asymptotic stability of the consensus model on the Stiefel manifold}

\author[D. Kim]{Dohyun Kim}
\address[D. Kim]{\newline Department of Mathematics Education, \newline Sungkyunkwan University, Seoul 03063, Republic of Korea}
\email{dohyunkim@skku.edu}

\author[W. Shim]{Woojoo Shim}
\address[W. Shim]{\newline Department of Mathematics Education, \newline Kyungpook National University, Republic of Korea}
\email{wjshim@knu.ac.kr}

\thanks{\textbf{Acknowledgment.} The work of D. Kim was supported by the National Research Foundation of Korea (NRF) grant funded by the Korean government (MSIT) (RS-2024-00454452). The work of W. Shim was supported by Kyungpook National University Research Fund, 2023.}

%Department of  Mathematics Education, Kyungpook National University , 
%
%Natural Science Bldg. Room 112, 80, Daehak-ro, Buk-gu, Daegu, Republic of Korea 41566. 
%
%Email: wjshim@knu.ac.kr

%The work of D. Kim  was supported by Sungshin Women's University research grant H20220021.}

%the National Research Foundation of Korea (NRF) grant funded by the Korean government (MSIT) (No.2021R1F1A1055929).

\begin{abstract}
%The aim of this article is to establish convergence toward equilibrium for the heterogeneous consensus model on the Stiefel manifold where heterogeneity is realized as distinct skew-symmetric matrices for each agent. Main difficulty comes from the fact that the model is no longer a gradient flow due to the heterogeneity.  Since there are not many machineries, we instead perform orbital stability to obtain the desired result. Our estimate improves the previous result in [Ha et al, Automatica {\bf136} (2022)]. Moreover, as a direct consequence of the asymptotic dynamics, we derive uniform-in-time stability with respect to initial data.
%

The aim of this article is to investigate the convergence properties of a heterogeneous consensus model on Stiefel manifolds. We consider each agent, without interaction, moving according to the flow determined by the fundamental vector field of the right multiplication action of the orthogonal group on the Stiefel manifold. We analyze the asymptotic behavior of $N$ such agents, assuming that, as a result of their interactions, each agent's velocity is the sum of its natural velocity and an additional velocity directed towards the average position of the $N$ agents. If the fundamental vector fields of all agents are the same, their movement can be represented as a gradient flow on a product manifold. In this study, we specifically investigate the asymptotic behavior in a non-gradient flow setting, where the fundamental vector fields are not all the same. Since fewer tools are available to address non-gradient flows, we perform an orbital stability analysis to obtain the desired results instead of relying on a gradient flow structure. Our estimate improves upon the previous result in [Ha et al., Automatica {\bf 136} (2022)]. Furthermore, as a direct consequence of the asymptotic dynamics, we derive uniform-in-time stability with respect to the initial data.

% 
% Precisely, we establish the uniform-in-time stability with respect to initial data when asymptotic consensus is a priori assumed. In this paper, we show that difference between any two solutions is uniformly-in-time controlled by their initial difference. As a direct consequence of the uniform stability, \textcolor{red}{to be filled later..}
 
\end{abstract}


\makeatletter
\@namedef{subjclassname@2020}{%
  \textup{2020} Mathematics Subject Classification}
\makeatother


\date{\today}

\maketitle

%\tableofcontents

%\noindent \textbf{Collective synchronization of the Kuramoto model has been extensively and intensively studied in diverse scientific disciplines since the 1970s. In this paper, we are concerned with the high-dimensional extension of the Kuramoto model equipped with higher-order couplings, in which a simple sinusoidal interaction is replaced by the sine function of multiples of the angle difference, and for the high-dimensional extension, the unit sphere and the unitary group are considered. Concerning the proposed models on the high-dimensional manifolds, cluster synchrony is of our main interest. More precisely, we develop sufficient frameworks leading to the cluste \section{A solvability of a matrix-valued Riccati differential equation} \label{sec:app.A}
\section{Introduction}
\setcounter{equation}{0}
 % 
% \section*{Data availability}
% The data that support the findings of this study are available from the corresponding author upon reasonable request.

Optimization problems on the Stiefel manifold \cite{C-M,J-D,Wang-Y,Zhu} have been extensively studied due to the reducibility of computational cost and their powerful applications. For instance, they are used in statistics \cite{C-V}, linear eigenvalue problems \cite{G-V,W-Y-L-Z}, finding the nearest low-rank correlation matrix \cite{L-Q}, singular value decomposition \cite{L-W-Z,S-I}, and applications to computer vision \cite{Lui,T-V-S-C}. See also \cite{E-A-S,S-S,W-Y} for optimization problems with orthogonality constraints. The Stiefel manifold $\St$ \cite{Stiefel} is defined as $\St := \{X \in \mathcal{M}_{n,p}(\mathbb{R}): X^\top X = I_p\}$, where $\mathcal{M}_{n,p}(\mathbb{R})$ is the set of all $n \times p$ matrices with real entries, $I_p$ is the $p \times p$ identity matrix, and $\top$ denotes the transpose of a matrix. We also denote $\|X\| := \sqrt{\text{tr}(X^\top X)}$  the (Frobenius) norm of a matrix $X \in \mathcal{M}_{n,p}(\mathbb{R})$.

\subsection{Model description}
We consider a consensus model on $\St$, described by a system of differential equations for the state ensemble $\mathcal{S} := (S_1, \ldots, S_N)$, where $S_1, \ldots, S_N$ represent $N$ elements of the Stiefel manifold:
\begin{equation} \label{main}
\begin{aligned}
&\dot S_i = S_i \Xi_i + \kappa \left( S_{ic} - \frac12( S_iS_i^\top S_{ic} + S_i S_{ic}^\top S_i) \right),\\
&S_{ic} :=\frac{1}{N}\sum_{k=1}^Na_{ik}S_k,\quad i\in [N]:=\{1,\ldots,N\},\\
&S_{i}(0)=S_i^\init \in \St,\quad i\in [N].
\end{aligned}
\end{equation} Here, $\kappa \in [0, \infty)$ is a non-negative constant measuring a uniform coupling strength between agents, and $(a_{ik})$ is a symmetric and connected network topology, with $a_{ik} = a_{ki}\geq 0$ for all $i,k\in [N]$. Moreover, $\Xi_i$ is a $p \times p$ skew-symmetric matrix that represents the natural frequency of agent $i$. We say that \eqref{main} is \textit{homogeneous} when $\Xi_i \equiv \Xi$, and \eqref{main} is \textit{heterogeneous} when $\Xi_i \neq \Xi_j$ for some $(i, j)$. In particular, the homogeneous ensemble with $\Xi\equiv O$ can be represented as a gradient flow with a total squared distance functional as its potential:
\begin{equation} \label{main-2}
\mathcal V(\mathcal S): = \frac{1}{N} \sum_{i,k=1}^N a_{ik} \|S_i- S_k\|^2.
\end{equation}
Since $\St$ is a compact manifold, it follows from standard literature, for instance, \L ojasiewicz inequality \cite{Loja}  that a solution to  \eqref{main} with $\Xi \equiv O$  converges to equilibrium  regardless of initial data. 

{ One notable feature of \eqref{main} is that when all $\Xi_1,\ldots,\Xi_N$ commutes with another skew-symmetric matrix $\Xi$, the dynamics of $\{\tilde{S}_i:=S_i\exp(-t\Xi)\}_{i=1}^N$ can be also represented as the same model with natural frequencies $\{\Xi_i-\Xi\}_{i=1}^N$:
	\begin{equation}\label{main-3}
	\begin{aligned}
	&\dot{\tilde{S}}_i=\tilde{S}_i \tilde{\Xi}_i+\kappa \left( \tilde{S}_{ic} - \frac12( \tilde{S}_i\tilde{S}_i^\top \tilde{S}_{ic} + \tilde{S}_i \tilde{S}_{ic}^\top \tilde{S}_i) \right),\\
	&\tilde{S}_i:=S_i\exp(-t\Xi),\quad \tilde{S}_{ic} :=\frac{1}{N}\sum_{k=1}^Na_{ik}\tilde{S}_k,\\
	&\tilde{\Xi}_i=\Xi_i-\Xi.
	\end{aligned}
	\end{equation}
	Therefore,  every homogeneous ensemble $\Xi_i\equiv \Xi$ can be viewed as a dynamics of $\Xi_i\equiv O$ observed in an appropriate moving frame, which we know its convergence as $t\to\infty$.} Since we can rewrite  \eqref{main} as 
\[
\dot S_i = u_i -\kp \nabla_{S_i} \mathcal V(\mathcal S),\quad u_i := S_i\Xi_i,
\]
where $u_i$ is the state-dependent control input for the $i$-th agent, \eqref{main} can be understood as a perturbed system of the gradient flow  in a moving frame by the external control. Thus, our natural goal is to verify whether convergence properties are robust to small perturbations when the effect of $u_i$ is smaller than the effect of $\kappa$.

\subsection{Main results}
The main results of this paper address the asymptotic stability of \eqref{main}. First, when the effect of the natural frequency matrices $\{\Xi_i\}_{i=1}^N$ is sufficiently small compared to the coupling strength $\kappa$, we observe that the composite matrix $S_i^\top S_j$, which can be interpreted as a matrix-valued inner product in the sense of a Hilbert module, converges to definite constant matrices for each $i, j \in [N]$. Before presenting the first main result, we define the concept of emergent behavior for \eqref{main}.

\begin{definition}
	For a solution $\mathcal S = \{S_1,\ldots, S_N \}$ to system \eqref{main}, we say that  system \eqref{main} exhibits asymptotic consensus if for each $i,j \in [N]$,
	\[
	\lim_{t\to\infty}  (S_i^\top S_j)(t) \quad \textup{exists}.
	\]
	In particular, if all $S_i^\top S_j$ converges to $I_p$, then we say that the system exhibits asymptotic complete consensus. 
\end{definition}

%Note that if there exist $k$ known matrices $\{S_1^\infty,\ldots,S_{k}^\infty\}$ in $\St$ such that  $\mbox{Im}(S_1^\infty)+\cdots+\mbox{Im}({S_k^\infty})=\mathbb{R}^n$, and if we know $(S_i^\infty)^T(S_{k+1}^\infty)$  for all $i=1,\ldots,k$, we can recover the unknown matrix $S_{k+1}^\infty$ since we know the inner products between an arbitrary vector in $\mathbb{R}^n$ and column vectors of $S_{k+1}^\infty$. %the asymptotic consensus implies the convergence of all $(S_i-S_j)^\top(S_i-S_j)=2I_p-S_i^\top S_j-S_j^\top S_i$, and
% This means that if $pN\geq n$, then in most cases there exists a unique configuration $(S_1^\infty,\ldots,S_N^\infty)$   satisfying
%\[(S_i^\infty)^\top (S_j^\infty)=S_{ij}^\infty\in \mathcal{M}_{p,p}(\mathbb{R}),\quad i,j\in [N] \]
%for given $S_{ij}^\infty$'s, up to left multiplication by a matrix in $O(n)$. In addition, from the relation
%\[\begin{aligned}
%&\mbox{rank}((S_i^\infty)^\top (S_j^\infty))\\
%&=p-\dim \mbox{ker}((S_i^\infty)^\top (S_j^\infty))\\
%&=p-\dim(\mbox{Im}(S_i^\infty)^\perp\cap \mbox{Im}(S_j^\infty))\\
%&=p-\dim((\mbox{Im}(S_i^\infty)\cap \mbox{Im}(S_j^\infty))^\perp\cap \mbox{Im}(S_j^\infty))\\
%&=\dim((\mbox{Im}(S_i^\infty)\cap \mbox{Im}(S_j^\infty)),
%\end{aligned} \]
%one can see that the matrix $(S_i^\infty)^\top (S_j^\infty)$ has to be singular matrix in $\mathcal{M}_{p,p}(\mathbb{R})$ in most cases. 

Heuristically, as expected and mentioned before, if the heterogeneity is relatively small compared to the coupling strength, and the initial positions are sufficiently close to each other, then the matrix-valued inner products converge to stationary states. In other words, asymptotic consensus arises. These assumptions are formalized in the framework $(\mathcal{F})$ below in Section \ref{sec:2.4}. Roughly speaking, when the coupling strength is large, the effect of the natural frequency matrices becomes negligible, the separable network topology approximates an all-to-all network, and the initial diameter is small. 

\begin{theorem} \label{T1.1}
Suppose that initial data and system parameters satisfy framework $(\mathcal F)$, and let $\mathcal S$ be a solution to \eqref{main}. Then, asymptotic consensus occurs. 
\end{theorem}
If there exist two constant ensembles $(S_1,\ldots,S_N),~(T_1,\ldots,T_N)\in \St^N$ satisfying 
\[S_i^\top S_j=T_i^\top T_j=:A_{ij},\quad i,j\in [N], \]
we are able to know the inner product between any two of column vectors of $S_i$'s and $T_i$'s, respectively. 
Therefore, we can inductively construct a linear isometry on $\mathbb{R}^n$, which maps each $k$-th column vector of $S_i$ to the corresponding $k$-th column vector of $T_i$, for all $1\leq k\leq p$ and $i\in [N]$. More precisely, we can find a constant orthogonal matrix $O\in O(n)$ such that 
\[OS_i=T_i,\quad i\in [N]. \]	
%	As each of these is a unit vector, this is equivalent to knowing the distance between any two points of the $S$-polytope (and $T$-polytope, respectively) when the column vectors are viewed as the positions of points in $\mathbb{R}^n$. In addition, since the (high-dimensional) volume of any polytope is completely determined by the distances between each vertex pairs, the dimenion of $S$-polytope and $T$-polytope are the same. Since each inner product between the vectors of $S$-column vectors are the same with the corresponding inner product between $T$-column vectors, there exists an appropriate isometry between those two subspaces which maps the $S$-polytope to the $T$-polytope.
From this observation, one can see that asymptotic consensus implies the convergence of the solution in some appropriate moving frame. Furthermore, this result can truly be regarded as an extension of the findings in \cite{HKK22} to a perturbed system. In \cite{HKK22}, it was shown that for a homogeneous ensemble with $\Xi_i\equiv O$, all $S_i$ converge to a same value  when all initial $S_i$'s are sufficiently close to each others. This implies that, for a general homogeneous ensemble $\Xi_i\equiv \Xi$ with sufficiently close initial $S_i$'s, all $\tilde{S}_i:=S_i\exp(-t\Xi)$ converge to a same value $S^\infty\in \St$ and 
\[\|S_i^\top S_j-I_p \|=\|\exp(t\Xi)^\top\tilde{S}_i^\top\tilde{S}_j\exp(t\Xi)-I_p \|=\|\tilde{S}_i^\top\tilde{S}_j-I_p \|\to \|(S^\infty)^\top S^\infty-I_p\|=0,  \]
which means the asymptotic complete consensus of the ensemble.

The second result is dedicated to uniform-in-time stability for system \eqref{main} whose definition is recalled below. For this, we  define the $\ell_p$-norm for  a set of matrices $\mathcal X: = \{X_1,\ldots,X_M\} \in \St^M$:
\[
\opnorm{\mathcal X}_p:= \left( \sum_{i=1}^M \|X_i\|^p \right)^\frac1p,\quad p\in [1,\infty).
\]
\begin{definition} \label{D1.2} 
For any two solutions $\mathcal S = \{S_1,\ldots,S_N\}$ and $\tilde{ \mathcal S }= \{\tilde S_1,\ldots,\tilde S_N\}$ to system \eqref{main} with their initial data $\mathcal S^\init$ and $\tilde{\mathcal S^\init}$, respectively, the system is said to be uniform-in-time  $\ell_p$-stable with respect to initial data if there exists a nonnegative constant $G$ independent of time $t$ and number of agents $N$ such that
\[
\sup_{0\leq t<\infty} \opnorm{ \mathcal S(t)  - \tilde {\mathcal S}(t)  } \leq G \opnorm{ \mathcal S^\init - \tilde {\mathcal S^\init} } .
\]
\end{definition}

\begin{theorem} \label{T1.2} 
Suppose that system exhibits asymptotic complete consensus a priori, and let $\mathcal S$ and $\tilde{\mathcal S}$ be any two solutions to system \eqref{main}. 
\begin{enumerate}
\item For any  network structure $(a_{ik})$, the system  is uniform-in time $\ell_1$-stable with respect to initial data in the sense of Definition \ref{D1.2}. 
\item If the network topology is separable in the sense that $a_{ik} = \xi_i\xi_k$ with $\xi_i>0$, then for any $p>0$, the system  is uniform-in time $\ell_p$-stable with respect to initial data in the sense of Definition \ref{D1.2}. 
\end{enumerate}
\end{theorem}

The proofs of the two main results are found in Section \ref{sec:3} and Section \ref{sec:4}, respectively, by using several key lemmas.  %Since the proofs of the lemmas are rather lengthy and consist of several straightforward calculation and tedious algebraic manipulation, such mathematical argument might distract readers and lead to misunderstanding. For this reason, all detailed proofs of lemmas are intentionally omitted and are found  in the expanded version \cite{KS23} of this paper only for interested readers. 

%For later use, we introduce diameters measuring the maximal distance between agents: for $t>0$, 
%\begin{equation} \label{A-10}
%\mathcal D(\mathcal S(t) ) := \max_{1\leq i,j\leq N} \|S_i(t) -S_j(t) \|, \quad \mathcal D(\tilde {\mathcal S}(t) ) := \max_{1\leq i,j\leq N} \|\tilde S_i(t) -\tilde S_j(t) \|.
%\end{equation}
%Furthermore, we set 
%\begin{equation} \label{A-11}
%\mathcal Z(t) := \max\{ \mathcal D( \mathcal S(t)), \mathcal D(\tilde{\mathcal S}(t))  \}.
%\end{equation}
%Lastly,
%\[
%d(\mathcal S,\tilde {\mathcal S} ) (t) : = \max_{1\leq i ,j\leq N } \|(S_j^\top S_i)(t) - (\tilde S_j^\top \tilde S_i)(t)\|
%\]

\subsection{Novelty and contribution} 
In literature, \eqref{main} for homogeneous case has been studied. To name a few, the authors in \cite{M1,M2} showed that the consensus manifold  defined $\mathcal C := \{ (S_i)_{i=}^N \in \St^N : S_i = S_j,~i,j\in [N]\}$ is almost globally asymptotically stable when $p \leq \frac{2n}{3} -1$. This result holds for generic initial data; however, there was  restriction on the pair $(p,n)$. On the other hand, the first author of this paper and his collaborators showed in \cite{HKK22} that \eqref{main} for homogeneous case exhibits asymptotic complete consensus for restricted initial data and any pair $(p,n)$. We would say that these two results are complementary.  For detail statements, we refer the reader to Section \ref{sec:2.3}.

Regarding the heterogeneous case, only a partial and weak result was provided in \cite{HKK22}. There, the sufficient initial condition leading to asymptotic consensus for the heterogeneous model depends on the number of agents $N$, and the size of the admissible initial data shrinks to zero as $N$ increases. In the current work, we overcome this restriction by employing a completely different method from \cite{HKK22}, which we call orbital stability analysis. Consequently, our initial configuration is independent of the number of agents $N$, achieved through a carefully conducted sharp analysis.

It is worthwhile to mention that \eqref{main} can be understood as a (small) perturbation of the gradient flow on the Stiefel manifold. In dynamical systems theory, one notable feature of a gradient flow (particularly on bounded manifolds) is that a solution converges to equilibrium. However, once the gradient structure is broken, such convergence is no longer guaranteed. In other words, it is unclear whether a solution converges to a stationary state. Since our system is perturbed by the introduction of heterogeneous skew-symmetric matrices, a natural question arises: if the heterogeneity is small, is the convergence property of a gradient flow preserved? Or does a solution still converge to equilibrium? Our answer to this question is affirmative in the sense that all matrices $S_i^\top S_j$ converge, but the solution $S_i$ itself may not. See also \cite{KP22} for similar problems in consensus models on the unit sphere and the unitary group.

Lastly, uniform-in-time stability was first considered in \cite{HKZ} for the Cucker-Smale flocking model \cite{C-S}. In fact, this stability can be applied to derive the mean-field limit when the number of agents is sufficiently large. Specifically, if $N \gg 1$, it becomes more effective to consider the temporal evolution of a probability density function using the BBGKY hierarchy. Note that the density function is governed by a kinetic-type partial differential equation (PDE). Therefore, we can estimate the distance between the measure-valued solution of the mean-field equation and the solution of the original model using uniform-in-time stability. Moreover, for some equilibrium ${S_i^\infty}$ of \eqref{main}, which is a trivial solution, if we set $\tilde{S}_i = S_i^\infty$, uniform-in-time stability also ensures the stability of the equilibrium.

\subsection{Notation}  
Before closing this section, several notations are introduced for later use.
\subsubsection{State matrix} We write relative correlation matrices  for any two solutions $\mathcal S$ and $\tilde{\mathcal S}$
\[
A_{ji} := S_j^\top S_i,\quad \tilde A_{ji} := \tilde S_j^\top \tilde S_i.
\]
Then, the maximal diameter for $\mathcal S$ is defined by 
\[
\mathcal D(\mathcal S) := \max_{1\leq i,j \leq N} \|S_i -S_j\|,
\]
and $\ell_2$ diameters for $\mathcal A = (A_{ji})$ are given by
\begin{align*}
&\opnorm{\mathcal A- \tilde{\mathcal A}}_2^2:= \sum_{i,j=1}^N \|A_{ji} - \tilde A_{ji}\|^2, \\
&  \opnorm{(\mathcal A - \mathcal A^\top) - (\tilde{\mathcal A} - \tilde{\mathcal A}^\top)   }_2^2 := \sum_{i,j=1}^N \| (A_{ji} -  A_{ji}^\top) - (\tilde A_{ji} - \tilde A_{ji}^\top)\|^2 .
\end{align*}
For simplicity, we also write $\mathcal D(\mathcal A) := \opnorm{\mathcal A- \tilde{\mathcal A}}_2^2 + \opnorm{(\mathcal A - \mathcal A^\top) - (\tilde{\mathcal A} - \tilde{\mathcal A}^\top)   }_2^2 $.
\subsubsection{System parameters}
We measure the heterogeneity for $\{\Xi_i\}$:
\[
\mathcal D(\Xi) := \max_{1\leq i,j\leq N } \|\Xi_i - \Xi_j\|,
\]
and some statistical quantities for $\{\xi_i\}$:
\[
\xi_m := \min_{1\leq i \leq N} \xi_i,\quad \xi_M := \max_{1\leq i \leq N} \xi_i,\quad \mathcal D(\xi) := \xi_M-\xi_m, \quad \xi_c := \frac1N \sum_{k=1}^N \xi_k.
\]

The rest of this paper is organized as follows. In Section \ref{sec:2}, we provide basic known results for the Stiefel manifold and the main model, and review previous results for relevant literature. Furthermore, the framework $(\mathcal F)$ and the strategy for the main results are introduced. Then, the proofs of two main results are provided in Section \ref{sec:3} and Section \ref{sec:4}, respectively. Finally, Section \ref{sec:5} is devoted to a brief summary of the main results. 

\vspace{0.3cm}

\section{Preliminaries} \label{sec:2} 
\setcounter{equation}{0}
In this section,  several preliminaries for main theorem are provided.

\subsection{The Stiefel manifold}
Since the Frobenius norm of each $X \in \St$ is $\sqrt p$, we know that the Stiefel manifold is compact whose dimension is $pn - \frac{p(p+1)}{2}$. Note that the category of Stiefel manifolds includes various well-known manifolds, such as the unit sphere $\textup{St}(1,n) = \mathbb{S}^{n-1}$, the special orthogonal group $\textup{St}(n-1,n) = \textup{SO}(n)$, the orthogonal group $\textup{St}(n,n) = \textup{O}(n)$, and so on.  If we specifically consider the manifold $\mathbb{S}^1=\textup{St}(1,2)$, the consensus model \eqref{main} precisely represents the dynamics of $\{e^{{\rm i}\theta_j} \}_{j=1}^N$'s when the phases $\{\theta_j\}_{j=1}^N$ follow the Kuramoto model.

\subsection{Dynamical Properties}
Once a dynamical system on a specific manifold is studied, the positivity of the governing manifold should be guaranteed. Although this has been shown in previous literature, such as \cite{HKK22}, we provide the proof here for the sake of consistency within the paper.

\begin{lemma} \label{L2.1}
Let $\mathcal S = \{S_1,\ldots,S_N\}$ be a solution to system \eqref{main} with initial data $\mathcal S^\init =\{S_1^\init, \ldots, S_N^\init\}$. Then, the states stay on the Stiefel manifold for all time, provided they are initially on the Stiefel manifold. In other words, if $S_i^\init \in \St$ for $i\in [N]$, then $S_i(t) \in \St$ for $i\in [N]$ and $t>0$.
\end{lemma}

\begin{proof} 
For simplicity, we write $H_i := I_p - S_i^\top S_i$. Then, $H_i$ satisfies
\[
\begin{aligned}
\frac\dd\dt H_i%&=-\dot{S}_i^\top S_i-S_i^\top\dot{S}_i\\
%&=-\left(S_i \Xi_i + \kappa \left( S_{ic} - \frac12( S_iS_i^\top S_{ic} + S_i S_{ic}^\top S_i) \right)\right)^\top S_i\\
%&\hspace{0.5cm}-S_i^\top\left(S_i \Xi_i + \kappa \left( S_{ic} - \frac12( S_iS_i^\top S_{ic} + S_i S_{ic}^\top S_i) \right)\right)\\
%&=\Xi_i S_i^\top S_i-\kappa \left(S_{ic}^\top S_i-\frac{1}{2}(S_{ic}^\top S_iS_i^\top S_i+S_i^\top S_{ic}S_i^\top S_i )\right)\\
%&\hspace{0.5cm}-S_i^\top S_i \Xi_i - \kappa \left(S_i^\top S_{ic} - \frac12( S_i^\top S_iS_i^\top S_{ic} + S_i^\top S_i S_{ic}^\top S_i) \right)\\
&=H_i \Xi_i-\Xi_i H_i-\frac{\kappa}{2}H_i(S_i^\top S_{ic}+S_{ic}^\top S_i)-\frac{\kappa}{2}(S_i^\top S_{ic}+S_{ic}^\top S_i)H_i,\\
\frac \dd\dt \|H_i\|^2&=\mbox{tr}\left(\dot{H_i}H_i+H_i\dot{H_i}\right)=-2\kappa\mbox{tr}(H_i^2(S_i^\top S_{ic}+S_{ic}^\top S_i))\leq 2\kappa \|H_i\|^2\|S_i^\top S_{ic}+S_{ic}^\top S_i\|.
\end{aligned}%B_i H_i   - H_iB_i,\quad B_i :=\Xi_i - \frac{\kp}{2N} \sum_{k=1}^N a_{ik} ( S_i^\top S_k + S_k^\top S_i).
\]
Define a temporal set 
\[
\mathcal T := \{ T>0: \max_{1\leq i \leq N } \|H_i(t)\| = 0,\quad t\in [0,T) \}.
\]
Since all $S_i^\init$ are contained in $\St$, we know that $\mathcal T$ is nonempty due to the continuity of a solution. Thus, we can define $T_*:= \sup \mathcal T>0$. Suppose to the contrary that $T_*<\infty$. By the definition we have
\[
\max_{1\leq i \leq N } \|H_i (T_*)\|>0.
\]
On the other hand, since $S_i(t) \in \St$ for $t\in [0,T_*)$, we have
\[
\frac\dd\dt \|H_i\| \leq C \|H_i\|,\quad t\in [0,T_*)
\]
for some constant $C=C(T_*)>0$. Then, it follows from Gr\"onwall's inequality that 
\[
\|H_i(t) \| \leq \|H_i(0)\|e^{Ct}, \quad t\in [0,T_*).
\]
In particular, $\max_{1\leq i \leq N} \|H_i(T_*)\| =0$ which contradicts. Hence, we have $T_* = \infty$. 
\end{proof}

Next, we recall how the orbital stability can be applied to the convergence of a time dependent function. %For convenience, we use the following form in \cite{K24}.

\begin{lemma} \label{conv} \cite{H-R16,K24}
	Let $Z\in \bbr^m$ be a uniformly bounded solution to the  autonomous differential equation
	\begin{equation}\label{A-50}
	\dot Z = F(Z),\quad t>0,\quad Z(0) =Z_0,
	\end{equation} 
	where $F$ is a continuously differentiable vector field. If there exists a constant  $C>0$  such that 
	\begin{equation*} \label{A-51}
	\|Z(t) - \tilde Z(t) \| \leq e^{-Ct} ,\quad t>0,
	\end{equation*} 
	for each $\tilde{Z}$ satisfying $\dot{\tilde{Z}}=F(\tilde{Z})$, then $Z(t)$ converges to a definite value: there exists a constant vector $Z_\infty \in \bbr^m$ such that
	\begin{equation} \label{A-52}
	\lim_{t\to\infty} Z(t)= Z_\infty.
	\end{equation} 
\end{lemma}

%\begin{proof}
%Since the proof can be found in \cite{H-R16}, we sketch the proof for simplicity. For any $T>0$, a shifted function $Z(t+T)$ also becomes a solution to \eqref{A-50} which is autonomous. Then, we have $\|Z(t+T)-Z(t)\|\leq e^{-CT}$. When we put $(t,T) = (n,1)$ for $n,m\in \bbz_+$, one finds 
%\[
%\|Z(n+1) - Z(n) \| \leq e^{-Cn},
%\]
%and for $(t,T) = (n,m)$ with $m \in \bbz_+$, the triangular inequality gives
%\begin{align*}
%&\|Z(n+m) - Z(n) \|\\
%& \leq \|Z(n+m)- Z(n+m-1) \| + \cdots + \|Z(n+1) - Z(n)\|  \\
%&\leq e^{-C(n+m-1)} + \cdots + e^{-Cn} \leq \frac{e^{-Cn}}{1-e^{-C}}.
%\end{align*}
%This shows that the discretized sequence $\{ Z(n)\}$ is a Cauchy sequence in the compact space in $\bbr^d$ due to the uniformly boundedness of $Z(t)$. Hence, we can find the desired limit $Z_\infty$ satisfying \eqref{A-52}.
%\end{proof}
%

\subsection{Literature review} \label{sec:2.3} 
For the consensus model (or high-dimensional Kuramoto model) on the Stiefel  manifold, there is not much available literature. Markdahl and his collaborators studied the corresponding homogeneous model in \cite{M1,M2}. Precisely, they performed stability analysis to show that for generic initial data and restricted pairs $(p,n)$, the consensus state is stable in some sense. On the other hand, the first author of this paper and his collaborators studied both homogeneous and heterogeneous models by using diameter analysis. Hence, their results hold for any pair $(p,n)$ but well-prepared class of initial data. 

\begin{theorem}
Let $\mathcal S$ be a global solution to \eqref{main}. 
\begin{enumerate}
\item \cite{M1,M2}
Suppose that the dimension pair $(p,n)$ satisfies $p\leq \frac{2n}{3}-1$, and $\mathcal G = (V,E)$ with a vertex set $V$ and an edge set $E$ is connected. Let $\mathcal S$ be a global solution to \eqref{main} with $\Xi\equiv O$ on the graph $\mathcal G$. Then, the consensus manifold  is almost globally asymptotically stable. 
\item \cite{HKK22} If $\Xi_i \equiv O$ and initial diameter $\mathcal D(\mathcal S^\init) <\sqrt2$, then the system exhibits  asymptotic complete consensus. 
\item \cite{HKK22} If $\Xi_i \neq \Xi_j$ for some $i\neq j$, initial diameter $\mathcal D(\mathcal S^\init) <\mathcal O(N^{-1})$ and $\kp> \mathcal O(N^2)$, then the system exhibits asymptotic consensus.  
\end{enumerate}
\end{theorem}

It should be mentioned that the initial framework in \cite{HKK22} for the results of the heterogeneous model crucially depends on $N$. Specifically, for large $N\gg1$, the sufficient initial diameter will shrink to zero, while the coupling strength diverges to infinity. These assumptions are too restrictive in the setting with a large number of agents. However, in this work, we overcome this limitation by providing a detailed analysis. Lastly, we refer the reader to \cite{K23} for the emergence of asymptotic complete consensus for \eqref{main} with homogeneity and control parameters, where the convergence rate is achieved either in finite time or with an algebraic rate.

\subsection{Framework and strategy descriptions} \label{sec:2.4}
We here introduce the framework $(\mathcal F)$ for \eqref{main} leading to asymptotic consensus. 

$\bullet$ ($\mathcal F1$ and $\mathcal F2$: Network topology)
\begin{equation}\label{F1}
(\mathcal F1):~~\xi_M^2 <4\xi_m \xi_c. \qquad (\mathcal F2):~~ \mathcal D(\xi) < \frac{\xi_m\xi_c}{3 \xi_M}.
\end{equation}
These conditions imply that variance of  $\{\xi_k\}$ is small and hence $\{\xi_k\}$ is close to the identical one. %In fact, without loss of generality, we may assume $\xi_c=1$ by rescaling. 

$\bullet$ ($\mathcal F3$: Coupling strength and natural frequencies)
\begin{equation} \label{F3}
	\begin{aligned}
		\frac{\mathcal D(\Xi)}{\kappa} <  (\xi_m\xi_c-3\xi_MD(\xi))\cdot \frac{2-\frac{1}{100p}}{\frac{80\xi_M^2p}{\xi_m^2}+2-\frac{1}{100p}}.
	\end{aligned}
\end{equation}
This condition says that $\kappa$ is sufficiently large compared to the variance of $\{\Xi_i\}$.

$\bullet$ ($\mathcal F4$: Initial data)
\begin{equation} \label{F4}
  \mathcal D(\mathcal S^\init)< \frac{\xi_m\xi_c-3\xi_MD(\xi)-\frac{D(\Xi)}{\kappa}}{10 \xi_M^2 \sqrt p}.
\end{equation}
 This conditions says that the initial diameter is small so that all inner products $S_i^\top S_j$ are close to the identity.\\

Since the proof consists of several steps, our strategy is briefly introduced for the readers' convenience.

$\bullet$ (Step A):   we derive the orbital stability estimate which can be successfully applied only when the maximal diameter should be small (see Lemma \ref{L3.1}). 

$\bullet$ (Step B): we indeed show the maximal diameter becomes small as we wish under the framework $(\mathcal F)$ (see Lemma \ref{L3.2} and Lemma \ref{L3.3}).

$\bullet$ (Step C): combining Steps A and B and Lemma \ref{conv}, the desired convergence is obtained.

\vspace{0.5cm}

Lastly, we end this section with elementary inequality. 
\begin{lemma} \label{L2.2}
Let $\veps = \veps(t)$ be a nonnegative integrable function on $(0,\infty)$ and   $y=y(t)$ be a nonnegative $C^1$-function satisfying
\[
\dot y \leq \veps(t) y,\quad t>0.
\]
Then, there exists a (uniform) constant $G>0$ such that
\[
y(t) \leq G y(0),\quad t>0.
\]
\end{lemma}

\begin{proof}
The proof directly follows from the method of an integrating factor:
\begin{align*}
y(t) \leq e^{\int_0^t \veps(s) \dd s }y(0) \leq e^{\int_0^\infty \veps(t) \dd t }y(0) =: Gy(0).
\end{align*}
\end{proof}

\section{Proof of Theorem \ref{T1.1}} \label{sec:3}
\setcounter{equation}{0}
In this section, we complete the proof of Theorem \ref{T1.1} by providing several lemmas.  First, differential inequality  for the diameter $\mathcal D(\mathcal A)$ is derived to obtain the orbital stability.

\begin{lemma} \label{L3.1}
Let $\mathcal S$ and $\tilde{\mathcal S}$ be any two solutions to \eqref{main}. Then, we have
\[
\frac\dd\dt \mathcal D(\mathcal A) \leq  -4 (\kp\xi_m\xi_c - \veps) \opnorm{\mathcal A- \tilde{\mathcal A}}_2^2 -\kp ( 4\xi_m \xi_c - \xi_M^2) \opnorm{(\mathcal A - \mathcal A^\top) - (\tilde{\mathcal A} - \tilde{\mathcal A}^\top)   }_2^2,
\]
where $\veps=\veps(t)$ is a quantity depending on time $t$ that will be made sufficiently small as we wish:
\[
\veps := 5 \kp \xi_M^2\sqrt p (\mathcal D(\mathcal S) + \mathcal D(\tilde{\mathcal S})) +3 \kp \xi_M \mathcal D(\xi) + \mathcal D(\Xi)
\]
\end{lemma}

Next, we find temporal evolution of the maximal diameter $\mathcal D(\mathcal S)$. 
\begin{lemma} \label{L3.2}
Let $\mathcal S$ be a solution to \eqref{main}. Then, the maximal diameter $\mathcal D(\mathcal S)$ satisfies
\[
\frac\dd\dt \mathcal D(\mathcal S) \leq -\frac{\kp \xi_m^2 }{2} \mathcal D(\mathcal S) + \frac{\kp \xi_m^2}{4}\mathcal D(\mathcal S)^3 +2\sqrt p \mathcal D(\Xi), \quad t>0.
\]
\end{lemma}

Lastly, under the framework $(\mathcal F)$, we show that the maximal diameter $\mathcal D(\mathcal S)$ can be made small by increasing $\kp$. 

\begin{lemma} \label{L3.3}
 Suppose that initial data and system parameters satisfy the framework $(\mathcal F)$, and let $\mathcal S$ be a solution to \eqref{main}. Then, we have
 \[
 \sup_{t>0}\mathcal D(\mathcal S(t))< \frac{\xi_m\xi_c-3\xi_MD(\xi)-\frac{D(\Xi)}{\kappa}}{10 \xi_M^2 \sqrt p}.
 \]
\end{lemma}

\noindent \textbf{(Proof of Theorem \ref{T1.1})}: We are now ready to provide the proof of Theorem \ref{T1.1}. It  follows from $(\mathcal F1)$ in \eqref{F1} that $4\xi_m \xi_c -\xi_M^2>0$. We use Lemma \ref{L3.3}, $(\mathcal F2)$ in $\eqref{F1}_2$ and $(\mathcal F3)$ in $\eqref{F3}_2$ that for $t>0$, 
\[
\sup_{t>0}5\kp \xi_M^2 \sqrt p ( \mathcal D(\mathcal S(t)) + \mathcal D(\tilde{\mathcal S}(t))) < \kappa(\xi_m\xi_c-3\xi_MD(\xi)-\frac{D(\Xi)}{\kappa}), \]
which gives
\[
\inf_{t>0}(\kappa\xi_m\xi_c - \veps) >0.
\]
Hence, Lemma \ref{L3.1} yields
\[
\frac\dd\dt \mathcal D(\mathcal A)  < - \delta \mathcal D(\mathcal A),\quad \delta:= \min \{ 4\inf_{t>0}(\kappa\xi_m\xi_c - \veps), \kappa(4\xi_m\xi_c-\xi_M^2) \}>0.
\]
In particular, 
\[
\| A_{ij} (t) - \tilde A_{ji} (t) \| \leq  \opnorm{\mathcal A (t) -\tilde{\mathcal A}(t)} \leq \mathcal D(\mathcal A^0) e^{-\delta t},\quad t>T_*.
\]
Finally, we use Lemma \ref{conv} to conclude that for each $i,j\in [N]$, there exists a constant matrix $A_{ji}^\infty \in \St$ such that
\[
\lim_{t\to\infty} S_j^\top S_i(t) = A_{ji}^\infty.
\]
This completes the proof.

\section{Proof of Theorem \ref{T1.2}} \label{sec:4} 
\setcounter{equation}{0}
Before we provide the proof of Theorem \ref{T1.1}, we first introduce a following key lemma. 

\begin{lemma} \label{L4.1}
Let $\mathcal S$ and $\tilde{\mathcal S}$ be any two solutions to \eqref{main}. Then, we have
\[
\frac\dd\dt \|S_i - \tilde S_i\| \leq \frac{\kp}{N}\sum_{k=1}^N  a_{ik} \|S_k - \tilde S_k\| - \frac{\kp}{N} \sum_{k=1}^N a_{ik} \|S_i - \tilde S_i\| + \frac{\kp\mathcal Z(t)}{N} \sum_{k=1}^N a_{ik} \|S_i -\tilde S_i\|
\]
where $\mathcal Z(t) := \max\{\mathcal D(\mathcal S(t)),\mathcal D(\tilde{\mathcal S}(t)) \}$. 
\end{lemma}

We now provide the proof of Theorem \ref{T1.2}. 

\vspace{0.5cm}

\noindent \textbf{(Proof of Theorem \ref{T1.2}(1))}:  For the first assertion, we sum the relation in Lemma \ref{L4.1} with respect to $i\in [N]$ to find

\begin{align*}
\frac\dd\dt \sum_{i=1}^N \|S_i - \tilde S_i\| &\leq \frac{\kp}{N} \sum_{i,k=1}^N a_{ik} \|S_k - \tilde S_k\| - \frac{\kp }{N} \sum_{i,k=1}^N a_{ik} \|S_i - \tilde S_i\| + \frac{\kp \mathcal Z(t)}{N} \sum_{i,k=1}^N a_{ik} \|S_i - \tilde S_i\| \\
& \leq \kp a_M \mathcal Z(t) \sum_{i=1}^N \|S_i -\tilde S_i\| 
\end{align*}
where $a_M$ is the maximum of $\{a_{ik}\}$. Since we assumed that asymptotic complete consensus is a priori achieved, $\mathcal Z(t)$ converges to zero exponentially. Note from \cite{HKK22} that the convergence rate is always exponential.  Thus, the proof directly follows from Lemma \ref{L2.2}.

\vspace{0.5cm}

\noindent \textbf{(Proof of Theorem \ref{T1.2}(2))}: For the second assertion, we assumed that the network satisfies the separability condition, i.e., $a_{ik}  = \xi_i\xi_k$. We simply denote $x_i := \|S_i - \tilde S_i\|.$ In Lemma \ref{L4.1}, by  multiplying both sides with $p\|S_i - \tilde S_i\|^{p-1}$, one finds 
\begin{align} \label{C-46}
\begin{aligned}
\frac\dd\dt \|S_i - \tilde S_i\|^p & \leq \frac{\kp p }{N} \sum_{k=1}^N \xi_i\xi_k \|S_k - \tilde S_k\| \|S_i - \tilde S_i\|^{p-1} - \frac{\kp p}{N} \sum_{k=1}^N \xi_i\xi_k \|S_i- \tilde S_i\|^p \\
&\hspace{0.5cm}+ \frac{\kp p \mathcal Z(t) }{N} \sum_{k=1}^N \xi_i\xi_k \|S_i - \tilde S_i\|^p.
\end{aligned}
\end{align}
We sum \eqref{C-46} with respect to $i\in [N]$ to obtain
\begin{align} \label{C-50}
\begin{aligned}
\frac\dd\dt \sum_{i=1}^N \|S_i - \tilde S_i\|^p & \leq \frac{\kp p}{N} \sum_{i,k=1}^N \xi_i\xi_k \|S_k - \tilde S_k\|\|S_i - \tilde S_i \|^{p-1}  - \frac{\kp p }{N} \sum_{i,k=1}^N\xi_i\xi_k \|S_i - \tilde S_i\|^p \\
&\hspace{0.5cm} + \frac{\kp p \mathcal Z(t) } {N} \sum_{i,k=1}^N \xi_i\xi_k \|S_i - \tilde S_i\|^p.
\end{aligned} 
\end{align}
In \eqref{C-50}, it suffices to show that
\[
\frac{\kp p}{N} \sum_{i,k=1}^N \xi_i\xi_k \|S_k - \tilde S_k\|\|S_i - \tilde S_i \|^{p-1}  - \frac{\kp p }{N} \sum_{i,k=1}^N \xi_i\xi_k \|S_i - \tilde S_i\|^p \leq 0
\]
which is rewritten in terms of $x_i$: 
\begin{align*}
\sum_{i,k=1}^N \xi_i \xi_k x_k x_i^{p-1} - \sum_{i,k=1}^N \xi_i \xi_k x_i^p \leq 0.
\end{align*} 
For this, we observe
\begin{align*}
\sum_{i,k=1}^N \xi_i \xi_k x_k x_i^{p-1} &= \left( \sum_{k=1}^N \xi_k x_k \right)\left( \sum_{i=1}^N \xi_i x_i^{p-1}\right)
\end{align*}
and by H\"older's inequality, 
\begin{align*}
&\sum_{k=1}^N \xi_k x_k \leq \left( \sum_{k=1}^N \xi_k x_k^p \right)^\frac1p \left( \sum_{k=1}^N \xi_k \right)^{1-\frac1p}, \quad \sum_{i=1}^N \xi_i x_i^{p-1} \leq \left( \sum_{i=1}^N \xi_i \right)^\frac1p \left( \sum_{i=1}^N \xi_i x_i^p \right)^{1-\frac1p} .
\end{align*}
Thus, we have
\begin{align*}
\sum_{i,k=1}^N \xi_i \xi_k x_k x_i^{p-1} - \sum_{i,k=1}^N \xi_i \xi_k x_i^p \leq \left(  \sum_{i=1}^N \xi_i \right)  \left( \sum_{i=1}^N\xi_i x_i^p \right) - \sum_{i,k=1}^N \xi_i\xi_k x_i^p =0
\end{align*} 
Therefore, \eqref{C-50} gives 
\begin{align*}
\frac\dd\dt \sum_{i=1}^N \|S_i - \tilde S_i\|^p \leq \kp p \mathcal Z(t) \xi_M^2 \sum_{i=1}^N \|S_i - \tilde S_i\|^p . 
\end{align*}
These complete the proof. 

\begin{remark}
Without any assumption on initial data, it follows from \eqref{C-46} that  there exists $C>0$ 
\[
\opnorm{\mathcal S(t) -  \tilde{\mathcal S}(t)} \leq e^{Ct} \opnorm{\mathcal S^\init- \tilde{\mathcal S}^\init}, \quad t>0.
\]
Since the right-hand side tends to infinity as $t\to\infty$, the estimate above is not uniform-in-time. However, if we impose some initial conditions, for instance, leading to asymptotic complete consensus, then we can make the estimate uniform-in-time. 
\end{remark}

 \section{Conclusion} \label{sec:5}
We have studied the asymptotic stability of the consensus model on the Stiefel manifold, also known as the high-dimensional Kuramoto model, where a natural frequency is introduced to make the model heterogeneous. With the introduction of the natural frequency, the model is no longer a gradient flow, and this heterogeneity leads to various emergent dynamics. In this paper, we focus on the emergence of asymptotic consensus and provide a sufficient framework that ensures it. As a direct consequence of this asymptotic behavior, we establish uniform-in-time stability, which can be applied to the mean-field setting when the number of agents is sufficiently large.

 \appendix
 \section{Proof of Lemmas} \label{sec:A}
\setcounter{equation}{0}

\subsection{Proof of Lemma \ref{L3.1}}
We provide the proof of Lemma \ref{L3.1} in which the differential inequality for $\mathcal D(\mathcal A)$ is derived. \newline

\noindent $\bullet$ (Step A): First, we derive a differential inequality for $\opnorm{\mathcal A - \tilde{\mathcal A}}_2^2$. Recall from \cite[Appendix A]{HKK21} that
\begin{align} \label{C-12-2}
\begin{aligned}
\frac{\dd}{\dd t} (A_{ji} - \tilde A_{ji}) &= (A_{ji} - \tilde A_{ji})\Xi_i - \Xi_j(A_{ji} - \tilde A_{ji} ) + \frac{\kp}{2N} \sum_{k=1}^N \mathcal J_{1k} \\
&\hspace{0.5cm}+\frac{\kp}{2N} \sum_{k=1}^N \Big(a_{ik}(( A_{jk} - \tilde A_{jk}) - (A_{kj} - \tilde A_{kj}))  \\
&\hspace{3cm}+ a_{jk}(( A_{ki} - \tilde  A_{ki}) - (A_{ik} - \tilde A_{ik})) \Big),
\end{aligned}
\end{align}
where $\mathcal J_{1k}$ is introduced in \eqref{C-12-3}. By multiplying $(A_{ji} - \tilde A_{ji})^\top$, we find
 \begin{align} \label{C-13}
\begin{aligned}
 &\frac12\frac{\dd}{\dd t} \|A_{ji} - \tilde A_{ji}\|^2  \\
 &\hspace{0.5cm}= \textup{tr} \big[ \{ (A_{ji} - \tilde A_{ji})\Xi_i - \Xi_j(A_{ji} - \tilde A_{ji} ) \}       (A_{ji} - \tilde A_{ji})^\top  \big] + \frac{\kp}{2N} \sum_{k=1}^N \textup{tr}[ \mathcal J_{1k}     (A_{ji} - \tilde A_{ji})^\top    ]  \\
 &\hspace{1cm}+ \frac{\kp}{2N} \sum_{k=1}^N a_{jk} \textup{tr}[ ( ( A_{ki} - \tilde A_{ki}) - (A_{ik} - \tilde A_{ik})    )   (A_{ji} - \tilde A_{ji})^\top ] \\
 &\hspace{1cm}+ \frac{\kp}{2N} \sum_{k=1}^N a_{ik}\textup{tr}[ ( ( A_{jk} - \tilde A_{jk}) - (A_{kj} - \tilde A_{kj})    )   (A_{ji} - \tilde A_{ji})^\top ]  \\
 &\hspace{0.5cm} = : \mathcal J_2 + \frac{\kp}{2N} \sum_{k=1}^N \mathcal J_{3k} + \frac{\kp}{2N} \sum_{k=1}^N \mathcal J_{4k} + \frac{\kp}{2N} \sum_{k=1}^N \mathcal J_{5k},
\end{aligned}
\end{align}
where $\mathcal J_{1k}$ is defined as 
\begin{align} \label{C-12-3}
\begin{aligned}
\mathcal J_{1k} &: = \big( a_{ik} ( A_{jk} - \tilde A_{jk}) - a_{jk}(A_{jk} A_{ji} - \tilde A_{jk} \tilde A_{ji}) \big)  \\
&\hspace{0.5cm}+ \big( a_{ik} ( A_{kj} - \tilde A_{kj}) - a_{jk}(A_{kj} A_ {ji} - \tilde A_{kj} \tilde A_{ji}) \big)  \\
&\hspace{0.5cm}+ \big( a_{jk} (A_{ki} - \tilde A_{ki}) - a_{ik} (A_{ji} A_{ki} - \tilde A_{ji} \tilde A_{ki} )        \big)  \\
&\hspace{0.5cm}+ \big(   a_{jk}(A_{ik} -\tilde A_{ik}) - a_{ik}(A_{ji} A_{ik} - \tilde A_{ji} \tilde A_{ik})     \big) \\
& = : \mathcal J_{1k,1} + \mathcal J_{1k,2} + \mathcal J_{1k,3} + \mathcal J_{1k,4}. 
\end{aligned}
\end{align}
Since $\mathcal J_{1k,j},~j=1,\ldots,4$ share a common structure, it suffices to consider $\mathcal J_{1k,1}$:
\begin{align} \label{C-12-4}
\begin{aligned}
\mathcal J_{1k,1}& = a_{ik} ( A_{jk} - \tilde A_{jk}) - a_{jk}(A_{jk} A_{ji} - \tilde A_{jk} \tilde A_{ji}) \\
& = -a_{jk}(A_{ji} - \tilde A_{ji}) + a_{jk}(I_p - \tilde A_{jk} )(A_{ji} - \tilde A_{ji}) + a_{jk}(A_{jk} - \tilde A_{jk}) (I_p - A_{ji})  \\
&\hspace{0.5cm}+ (a_{ik} - a_{jk})(A_{jk} - \tilde A_{jk}).
\end{aligned}
\end{align}
$\diamond$ (Estimate of $\mathcal J_{2}$): since $\Xi_i$ and $\Xi_j$ are skew-symmetric matrices, we easily find 
\[
\mathcal J_2 = 0.
\]
$\diamond$ (Estimate of $\mathcal J_{3k}$): we first consider $\mathcal J_{1k,1}$:
\begin{align*}
&\textup{tr}[  \mathcal J_{1k,1}(A_{ji} - \tilde A_{ji})^\top  ]\\
&   \leq  -a_{jk} \|A_{ji} - \tilde A_{ji}\|^2 + a_{jk}\sqrt{p} \mathcal D(\tilde{\mathcal S}) \|A_{ji} -\tilde A_{ji}\|^2 \\
&  \hspace{0.5cm} + a_{jk}\sqrt{p}\mathcal D( \mathcal S) \|A_{jk} -\tilde A_{jk}\|\cdot \|A_{ji} - \tilde A_{ji}\|\\
&  \hspace{0.5cm} +  (a_{ik} - a_{jk}) \textup{tr} (( A_{jk} - \tilde A_{jk})(A_{ji} - \tilde A_{ji})^\top ) \\
& \leq  -a_{jk} \|A_{ji} - \tilde A_{ji}\|^2 + a_{jk}\sqrt{p} \mathcal D(\tilde{\mathcal S}) \|A_{ji} -\tilde A_{ji}\|^2 \\
&  \hspace{0.5cm} + ( a_{jk}\sqrt{p}\mathcal D( \mathcal S) + |a_{ik} - a_{jk}|)\|A_{jk} -\tilde A_{jk}\|\cdot \|A_{ji} - \tilde A_{ji}\|,
\end{align*}
where we used the following  inequality:
\begin{equation*}
\|I_p - A_{ji} \| = \|S_j^\top S_j - S_j^\top  S_i\|\leq \|S_j \|\cdot \|S_j-S_i\|  \leq \sqrt p D(\mathcal S).
\end{equation*}
Hence, $\mathcal J_{3k}$ can be estimated as  
\begin{align*}
\mathcal J_{3k}&  \leq  \textup{tr}[ (\mathcal J_{1k,1}  +\cdots +  \mathcal J_{1k,4})(A_{ji} - \tilde A_{ji})^\top ]  \\
& \leq -2(a_{ik} + a_{jk}) \|A_{ji} - \tilde A_{ji}\|^2 + 2\sqrt{p}(a_{ik} + a_{jk})\mathcal D( \tilde{\mathcal S})\|A_{ji} - \tilde A_{ji}\|^2 \\
&\hspace{0.5cm} + 2 (a_{jk}\sqrt{p}\mathcal D(\mathcal S) + |a_{ik}-a_{jk}|)\|A_{jk} - \tilde A_{jk}\|\cdot \|A_{ji} - \tilde A_{ji}\| \\
&\hspace{0.5cm} + 2 (a_{ik}\sqrt{p}\mathcal D(\mathcal S) + |a_{ik}-a_{jk}|)\|A_{ik} - \tilde A_{ik}\|\cdot \|A_{ji} - \tilde A_{ji}\| .
\end{align*}
$\diamond$ (Estimate of $\mathcal I_{4k}$): we consider
\[
\frac{\kp}{2N}\sum_{i,j,k=1}^N \mathcal J_{4k} = \frac{\kp}{2N} \sum_{i,j,k=1}^N a_{jk} \textup{tr}[ ( ( A_{ki} - \tilde A_{ki}) - (A_{ik} - \tilde A_{ik})    )   (A_{ji} - \tilde A_{ji})^\top ].
\]
For notational simplicity, we write
 \[
B_{ki} := \xi_k A_{ki},\quad B_{ci} := \frac1N \sum_{k=1}^N B_{ki},\quad A_{ci} := \frac1N \sum_{k=1}^N A_{ki} .
\]
We observe
\begin{align*}
\frac{\kp}{2N}\sum_{i,j,k=1}^N \mathcal J_{4k} & = \frac{\kappa N}{2} \sum_{i=1}^N \textup{tr}[     (B_{ci} - \tilde B_{ci})(B_{ci} - \tilde B_{ci})^\top - (B_{ci} - \tilde B_{ci})^2 ] \\
& = \frac {\kappa N}4 \sum_{i=1}^N \| (B_{ci} - \tilde B_{ci})- (B_{ci} - \tilde B_{ci})^\top \|^2\\
&\leq \frac \kappa 4 \sum_{i,j=1}^N \| (B_{ji} - \tilde B_{ji} ) - (B_{ji} - \tilde B_{ji})^\top\|^2 \\
&\leq  \frac {\kappa \xi_M^2}{4} \sum_{i,j=1}^N \| (A_{ji} - \tilde A_{ji} ) - (A_{ji} - \tilde A_{ji})^\top\|^2.
\end{align*}
$\diamond$ (Estimate of $\mathcal J_{5k}$): By using the exactly same way, 
\[
\frac{\kp}{2N}\sum_{i,j,k=1}^N \mathcal J_{5k} =\frac {\kappa N}4 \sum_{i=1}^N \| (B_{ci} - \tilde B_{ci})- (B_{ci} - \tilde B_{ci})^\top \|^2
\]
Note that
\[
|a_{ik} - a_{jk}| = |\xi_k| |\xi_i - \xi_j| \leq \xi_M \mathcal D(\xi)
\]
and
\[
\sum_{i,j,k=1}^N \|A_{jk} - \tilde A_{jk}\|\cdot  \| A_{ji} - \tilde A_{ji}\| \leq N \sum_{i,j=1}^N \|A_{ji} - \tilde A_{ji}\|^2.
\]
To this end,  we obtain 
\begin{align} \label{F-30}
\begin{aligned}
\frac12\frac\dd\dt \sum_{i,j=1}^N \|A_{ji} - \tilde A_{ji}\|^2 & \leq -\frac\kp N \sum_{i,j,k=1}^N (a_{ik} + a_{jk}) \|A_{ji} - \tilde A_{ji}\|^2  \\
& \hspace{0.5cm}  + 2\kp ( \xi_M^2 \sqrt p(\mathcal D(\mathcal S)  + \mathcal D(\tilde{\mathcal S})) + \xi_M \mathcal D(\xi))  \sum_{i,j=1}^N \|A_{ji} - \tilde A_{ji}\|^2 \\
& \hspace{0.5cm}  +\frac {\kappa N}2 \sum_{i=1}^N \| (B_{ci} - \tilde B_{ci})- (B_{ci} - \tilde B_{ci})^\top \|^2
\end{aligned}
\end{align}
%where the last term can be estimated as 
%\[
%\frac N 4 \sum_{i=1}^N \|(B_{ci}- \tilde B_{ci}) - (B_{ci} - \tilde B_{ci})^\top \|^2  \leq  \frac{\xi_M}{2}  \sum_{i,j=1}^N \| (A_{ji} - \tilde A_{ji} ) - (A_{ji} - \tilde A_{ji})\|^2.
%\]
\noindent $\bullet$ (Step B): Next, we are concerned with $\opnorm{(\mathcal A - \mathcal A^\top) - (\tilde{\mathcal A} - \tilde{\mathcal A}^\top)}_2^2$. Recall that 
\[
\begin{aligned}
\frac 12\|A_{ji} - \tilde A_{ji} - (A_{ji} - \tilde A_{ji})^\top \|^2 
&= \|A_{ji} - \tilde A_{ji} \|^2 -\frac{1}{2} \textup{tr} \left[(A_{ji} - \tilde A_{ji})^2+((A_{ji} - \tilde A_{ji})^2)^\top\right]\\
&=\|A_{ji} - \tilde A_{ji} \|^2-\textup{tr}\left[( A_{ji}-\tilde A_{ji})^2\right].
\end{aligned}
\]
By multiplying $A_{ji} - \tilde A_{ji}$, taking the trace, and closely following the estimate above, we find
\begin{align} \label{F-35}
\begin{aligned}
\frac12\frac\dd\dt \sum_{i,j=1}^N -\textup{tr} \left[( A_{ji}-\tilde A_{ji})^2\right]  & \leq  \frac\kp N \sum_{i,j,k=1}^N (a_{ik} + a_{jk}) \textup{tr} \left[( A_{ji} - \tilde A_{ji})^2\right]  + \mathcal D( \Xi) \sum_{i,j=1}^N \|A_{ji} - \tilde A_{ji}\|^2 \\
&  \hspace{0.5cm} + 2\kp \xi_M^2 \sqrt p ( \mathcal D( \mathcal S) + \mathcal D( \tilde{\mathcal S})) \sum_{i,j=1}^N \|A_{ji} - \tilde A_{ji}\|^2 \\
&  \hspace{0.5cm} - \frac {\kappa N}{2} \sum_{i=1}^N \|(B_{ci}- \tilde B_{ci}) - (B_{ci} - \tilde B_{ci})^\top \|^2 .
\end{aligned}
\end{align}
 Combining \eqref{F-30} and \eqref{F-35}, we derive 
\begin{align*}
&\frac14 \frac\dd\dt \sum_{i,j=1}^N \| (A_{ji} - \tilde A_{ji}) - (A_{ji} - \tilde A_{ji})^\top \|^2  \leq -\frac{\kp}{2N}  \sum_{i,j,k=1}^N (a_{ik} + a_{jk}) \| (A_{ji} - \tilde A_{ji}) - (A_{ji} - \tilde A_{ji})^\top \|^2 \\
& +(  4\kp \xi_M^2 \sqrt p (\mathcal D(\mathcal S) + \mathcal D(\tilde{\mathcal S})) + 2\kp\xi_M\mathcal D(\xi) + \mathcal D(\Xi)       )  \sum_{i,j=1}^N \|A_{ji} - \tilde A_{ji}\|^2 
\end{align*}
which gives
\begin{align} \label{F-40}
\begin{aligned}
&\frac14 \frac\dd\dt \sum_{i,j=1}^N \| (A_{ji} - \tilde A_{ji}) - (A_{ji} - \tilde A_{ji})^\top \|^2  \leq - \kp \xi_m \xi_c \sum_{i,j=1}^N  \| (A_{ji} - \tilde A_{ji}) - (A_{ji} - \tilde A_{ji})^\top \|^2 \\
& +(  4\kp \xi_M^2 \sqrt p (\mathcal D(\mathcal S) + \mathcal D(\tilde{\mathcal S})) + 2\kp \xi_M\mathcal D(\xi) + \mathcal D(\Xi)       )  \sum_{i,j=1}^N \|A_{ji} - \tilde A_{ji}\|^2 .
\end{aligned}
\end{align}
 Lastly, we recall
 \begin{align*}
\frac14\frac\dd\dt \sum_{i,j=1}^N \|A_{ji} - \tilde A_{ji}\|^2 & \leq -\frac{\kp}{2 N }\sum_{i,j,k=1}^N (a_{ik} + a_{jk}) \|A_{ji} - \tilde A_{ji}\|^2  \\
& +   (\kp \xi_M^2 \sqrt p(\mathcal D(\mathcal S)  + \mathcal D(\tilde{\mathcal S})) + \kp \xi_M \mathcal D(\xi))  \sum_{i,j=1}^N \|A_{ji} - \tilde A_{ji}\|^2 \\
& + \frac {\kappa \xi_M^2}{4} \sum_{i,j=1}^N \| (A_{ji} - \tilde A_{ji} ) - (A_{ji} - \tilde A_{ji})^\top\|^2,
\end{align*}
which also gives
 \begin{align} \label{F-45}
 \begin{aligned}
\frac14\frac\dd\dt \sum_{i,j=1}^N \|A_{ji} - \tilde A_{ji}\|^2 & \leq -\kp\xi_m \xi_c \sum_{i,j=1}^N  \|A_{ji} - \tilde A_{ji}\|^2  \\
& +   (\kp \xi_M^2 \sqrt p(\mathcal D(\mathcal S)  + \mathcal D(\tilde{\mathcal S})) + \xi_M \mathcal D(\xi))  \sum_{i,j=1}^N \|A_{ji} - \tilde A_{ji}\|^2 \\
& + \frac {\kappa \xi_M^2}{4} \sum_{i,j=1}^N \| (A_{ji} - \tilde A_{ji} ) - (A_{ji} - \tilde A_{ji})^\top\|^2.
\end{aligned}
\end{align}
$\bullet$ (Step C): Finally, we add \eqref{F-40} and \eqref{F-45} to obtain
\[
\frac\dd\dt \mathcal D(\mathcal A) \leq  -4 (\kp \xi_m\xi_c - \veps) \opnorm{\mathcal A- \tilde{\mathcal A}}_2^2 -\kp ( 4\xi_m \xi_c - \xi_M^2) \opnorm{(\mathcal A - \mathcal A^\top) - (\tilde{\mathcal A} - \tilde{\mathcal A}^\top)   }_2^2.
\]

\subsection{Proof of Lemma \ref{L3.2}}
By using \cite[Lemma 3.3]{HKK21}, we have
\begin{align} \label{C-0-3}
\begin{aligned}
&\frac{\dd}{\dd t} \|S_i - S_j\|^2 \\
&\hspace{0.5cm} \leq -\frac{\kp}{N} \sum_{k=1}^N a_{ik} ( \|S_i - S_j\|^2 - \|S_j - S_k\|^2) - \frac\kp N \sum_{k=1}^N a_{jk}( \|S_i- S_j\|^2 - \|S_i - S_k\|^2) \\
& \hspace{0.96cm}-( 2- \|S_ i- S_j\|^2) \cdot \frac{\kp}{2N } \sum_{k=1}^N( a_{ik} \|S_ i- S_k\|^2 + a_{jk} \|S_j - S_k\|^2)\\
& \hspace{0.96cm} + \textup{tr}(( A_{ji} - A_{ij})(\Xi_i -\Xi_j))\\
&\hspace{0.5cm}\leq  -\frac{\kp}{N} \sum_{k=1}^N \xi_m^2 ( \|S_i - S_j\|^2 - \|S_j - S_k\|^2+\|S_i-S_k\|^2)\\
&\hspace{0.96cm} - \frac\kp N \sum_{k=1}^N \xi_m^2( \|S_i- S_j\|^2 - \|S_i - S_k\|^2+ \|S_j - S_k\|^2)\\
&\hspace{0.96cm} +\frac{\kappa}{N}\sum_{k=1}^N\xi_m^2\|S_i-S_j\|^2( \|S_ i- S_k\|^2 +  \|S_j - S_k\|^2)\\
& \hspace{0.96cm} + \textup{tr}(( A_{ji} - A_{ij})(\Xi_i -\Xi_j)).
\end{aligned}
\end{align}
Then, Lemma \ref{L3.2} follows from the estimate
\[
|\textup{tr}(( A_{ji} - A_{ij})(\Xi_i -\Xi_j))| \leq 2\sqrt p \mathcal D(\Xi).
\]

\subsection{Proof of Lemma \ref{L3.3}}

Next, we provide the proof of Lemma \ref{L3.3}. We write
\[
\frac\dd\dt \mathcal D(\mathcal S) \leq \frac{\kp\xi_m^2}{4} f(\mathcal D(\mathcal S)), \quad f(r) := r^3 - 2r + \frac{8\sqrt p \mathcal D(\Xi)}{\kp \xi_m^2},\quad r\geq 0.
\]
By straightforward calculation, $f$ attains positive values at $x=0$ and $x=\sqrt{2}$.  Thus, if we verify 
\begin{equation}\label{A.10}
	f\left(\frac{\xi_m\xi_c-3\xi_MD(\xi)-\frac{D(\Xi)}{\kappa}}{10 \xi_M^2 \sqrt p}\right) <0,
\end{equation}
there are two positive roots $r_1,r_2$ of $f$ satisfying 
\[0<r_1<\frac{\xi_m\xi_c-3\xi_MD(\xi)-\frac{D(\Xi)}{\kappa}}{10 \xi_M^2 \sqrt p}<r_2<\sqrt{2}.\]
Therefore, whenever $D(\mathcal S)$ is contained in $(r_1,\frac{\xi_m\xi_c-3\xi_MD(\xi)-\frac{D(\Xi)}{\kappa}}{10 \xi_M^2 \sqrt p})$, its derivative $\frac{dD(\mathcal S)}{dt}$ must be strictly negative, and the set $\{ t>0:\mathcal D(\mathcal S) < \frac{\xi_m\xi_c-3\xi_MD(\xi)-\frac{D(\Xi)}{\kappa}}{10 \xi_M^2 \sqrt p}\}$ becomes positively invariant.\\

To show the \eqref{A.10}, we first note that the following inequality holds:

\[0<\frac{\xi_m\xi_c-3\xi_MD(\xi)-\frac{D(\Xi)}{\kappa}}{10 \xi_M^2 \sqrt p}\leq \frac{1}{10\sqrt{p}}. \]
Therefore, we have 
\[\begin{aligned}
	&f\left(\frac{\xi_m\xi_c-3\xi_MD(\xi)-\frac{D(\Xi)}{\kappa}}{10 \xi_M^2 \sqrt p}\right)\\
	&\hspace{0.5cm}\leq \left(\frac{\xi_m\xi_c-3\xi_MD(\xi)-\frac{D(\Xi)}{\kappa}}{10 \xi_M^2 \sqrt p}\right)(\frac{1}{100p}-2)+\frac{8\sqrt p \mathcal D(\Xi)}{\kp \xi_m^2}\\
	&\hspace{0.5cm}=\left(\frac{\xi_m\xi_c-3\xi_MD(\xi)}{10 \xi_M^2 \sqrt p}\right)(\frac{1}{100p}-2)+\frac{D(\Xi)}{\kappa}\left(\frac{2-\frac{1}{100p}}{10\xi_M^2\sqrt{p}}+\frac{8\sqrt{p}}{\xi_m^2} \right)\\
	&\hspace{0.5cm}=\frac{1}{10\xi_M^2\sqrt{p}}\left(-(2-\frac{1}{100p})(\xi_m\xi_c-3\xi_MD(\xi))+\frac{D(\Xi)}{\kappa}(2-\frac{1}{100p}+\frac{80\xi_M^2p}{\xi_m^2} ) \right)\\
	&\hspace{0.5cm}<0,
\end{aligned} \]
where we used $(\mathcal{F}3)$ in the last inequality.

\subsection{Proof of Lemma \ref{L4.1}}
The goal of this subsection is to derive the estimate of $\|S_i - \tilde S_i\|^2$ when $\{S_i\}$ and $\{\tilde S_i\}$ are solutions to \eqref{main}. For this,  we consider the difference $S_i-\tilde S_i$ and its transpose $S_i^\top - \tilde S_i^\top$:
\begin{align*}
\frac\dd\dt (S_i- \tilde S_i) &= (S_i - \tilde S_i)\Xi_i \\
& + \frac\kp N \sum_{k=1}^N a_{ik} \left[ (S_k - \tilde S_k) - \frac12 ( S_i S_i^\top S_k - \tilde S_i \tilde S_i^\top \tilde S_k ) - \frac12 ( S_iS_k^\top S_i - \tilde S_i \tilde S_k^\top \tilde S_i) \right]
\end{align*} 
and
\begin{align*}
\frac\dd\dt (S_i^\top - \tilde S_i^\top ) &= - \Xi_i(S_i^\top -\tilde S_i^\top) \\
&+  \frac\kp N \sum_{k=1}^N a_{ik} \left[  (S_k^\top  - \tilde S_k^\top ) - \frac12 ( S_k^\top S_i S_i^\top  -  \tilde S_k^\top \tilde S_i \tilde S_i^\top   ) - \frac12 ( S_i^\top S_k S_i^\top - \tilde S_i^\top \tilde S_k  \tilde S_i^\top ) \right].
\end{align*} 
Then, we observe 
\begin{align} \label{C-15}
\begin{aligned}
\frac12\frac\dd\dt \|S_i - \tilde S_i\|^2  &=  \frac{1}{2}\frac\dd\dt \tr ( (S_i^\top  - \tilde S_i^\top) ( S_i - \tilde S_i) )  \\
& = \frac{1}{2}\tr( ( S_i^\top - \tilde S_i^\top)(\dot S_i - \dot{\tilde S}_i))  +  \frac{1}{2} \tr (( \dot S_i^\top - \dot{ \tilde S}_i^\top) (S_i - \tilde S_i)) \\
& =: \mathcal I_1 + \mathcal I_2.
\end{aligned}
\end{align}
Note that the term containing $\Xi_i$ vanishes due to the skew-symmetricity of $\Xi_i$. 
Below, we calculate $\mathcal I_1$ and $\mathcal I_2$ in  respectively. \newline

\noindent $\bullet$ (Calculation of $\mathcal I_1$):  we first observe
\begin{align} \label{C-18}
\begin{aligned}
2\mathcal I_1 & = \tr( ( S_i^\top - \tilde S_i^\top)(\dot S_i - \dot{\tilde S}_i)) \\
& = \frac\kp N \sum_{k=1}^N a_{ik} \tr( (S_i^\top - \tilde S_i^\top ) (S_k - \tilde S_k) )\\
&\hspace{0.5cm} -\frac{\kp}{2N} \sum_{k=1}^N a_{ik} \tr( ( S_i^\top - \tilde S_i^\top) ( S_i S_i^\top S_k - \tilde S_i \tilde S_i^\top \tilde S_k) ) \\
&\hspace{0.5cm} -\frac{\kp}{2N} \sum_{k=1}^N a_{ik} \tr( ( S_i^\top - \tilde S_i^\top) (S_iS_k^\top S_i - \tilde S_i \tilde S_k^\top \tilde S_i)) \\
& =: \mathcal I_{11} + \mathcal I_{12} + \mathcal I_{13}.
\end{aligned}
\end{align}
For further calculation, we split $\mathcal I_1$ into three terms $\mathcal I_{1k},~k=1,2,3$. \newline 

$\diamond$ (Estimate of $\mathcal I_{11}$): we find
\begin{equation} \label{C-19}
\mathcal I_{11} = \frac\kp N \sum_{k=1}^N a_{ik} \tr( (S_i^\top - \tilde S_i^\top ) (S_k - \tilde S_k) )  \leq \frac\kp N \sum_{k=1}^N |a_{ik}| \|S_i - \tilde S_i\| \|S_k - \tilde S_k \|.
\end{equation}  

$\diamond$ (Calculation of $\mathcal I_{12}$): we observe 
\begin{align} \label{C-20}
\begin{aligned}
S_i S_i^\top S_k - \tilde S_i \tilde S_i^\top \tilde S_k & = S_i (S_i^\top S_k - \tilde S_i^\top \tilde S_k ) + (S_i - \tilde S_i) \tilde S_i^\top \tilde S_k \\
&= S_i (S_i^\top S_k - \tilde S_i^\top \tilde S_k ) + ( S_i - \tilde S_i) + (S_i - \tilde S_i)(\tilde S_i^\top \tilde S_k -I_p).
\end{aligned}
\end{align}
Then, $\mathcal I_{12}$ becomes
\begin{align*}
\mathcal I_{12} & =  -\frac{\kp}{2N} \sum_{k=1}^N a_{ik} \tr( ( S_i^\top - \tilde S_i^\top) ( S_i S_i^\top S_k - \tilde S_i \tilde S_i^\top \tilde S_k) ) \\
& = -\frac{\kp}{2N} \sum_{k=1}^N a_{ik} \tr( ( S_i ^\top - \tilde S_i^\top) S_i ( S_i^\top S_k - \tilde S_i^\top \tilde S_k)) - \frac{\kp}{2N} \sum_{k=1}^N a_{ik} \tr(( S_i^\top - \tilde S_i^\top) (S_i - \tilde S_i)) \\
 & \hspace{0.5cm}-\frac{\kp}{2N} \sum_{k=1}^N a_{ik} \tr(( (S_i^\top -\tilde S_i^\top)(S_i - \tilde S_i) (\tilde S_i^\top \tilde S_k - I_p))) \\
 & =: \mathcal I_{121} + \mathcal I_{122} + \mathcal I_{123} 
\end{align*} 
Again, we separate $\mathcal I_{12}$ into three terms $\mathcal I_{12k},~k=1,2,3$. \newline  

$\circ$ (Calculation of $\mathcal I_{122}$): by definition of the Frobenius norm,
\[
\mathcal I_{122} = - \frac{\kp}{2N} \sum_{k=1}^N a_{ik} \tr(( S_i^\top - \tilde S_i^\top) (S_i - \tilde S_i)) = -\frac{\kp}{2N} \sum_{k=1}^N a_{ik} \|S_i - \tilde S_i\|^2 .
\]

$\circ$ (Estimate of $\mathcal I_{123}$): we find
\begin{align*}
\mathcal I_{123}  & = -\frac{\kp}{2N} \sum_{k=1}^N a_{ik} \tr(( (S_i^\top -\tilde S_i^\top)(S_i - \tilde S_i) (\tilde S_i^\top \tilde S_k - I_p))) \\
& \leq \frac{\kp}{2N}\sum_{k=1}^N |a_{ik}| \mathcal Z(t) \|S_i - \tilde S_i\|^2.
\end{align*}
Thus, $\mathcal I_{12}$ satisfies 
\begin{equation} \label{C-25}
\mathcal I_{12}  \leq \mathcal I_{121} -\frac{\kp}{2N} \sum_{k=1}^N a_{ik} \|S_i - \tilde S_i\|^2  + \frac{\kp \mathcal Z(t) }{2N}\sum_{k=1}^N |a_{ik}|  \|S_i - \tilde S_i\|^2.
\end{equation}

$\diamond$ (Estimate of $\mathcal I_{13}$): similarly for \eqref{C-20} in the estimate of $\mathcal I_{12}$, we use 
\begin{align*}
S_iS_k^\top S_i - \tilde S_i \tilde S_k^\top \tilde S_i & = S_i ( S_k^\top S_i - \tilde S_k^\top \tilde S_i) + (S_i  - \tilde S_i) \tilde S_k^\top \tilde S_i  \\
 & = S_i ( S_k^\top S_i - \tilde S_k^\top \tilde S_i) + (S_i  - \tilde S_i) + (S_i - \tilde S_i) (\tilde S_k^\top \tilde S_i - I_p)
\end{align*}
to find 
\begin{align*}
\mathcal I_{13} &  =  -\frac{\kp}{2N} \sum_{k=1}^N a_{ik} \tr( ( S_i^\top - \tilde S_i^\top) (S_iS_k^\top S_i - \tilde S_i \tilde S_k^\top \tilde S_i)) \\
& = -\frac{\kp}{2N} \sum_{k=1}^N a_{ik} \tr(( S_i^\top - \tilde S_i^\top)S_i(S_k^\top S_i - \tilde S_k^\top \tilde S_i)) -\frac{\kp}{2N} \sum_{k=1}^N a_{ik} \tr(( S_i^\top -\tilde S_i^\top)(S_i- \tilde S_i)) \\
& \hspace{0.5cm}- \frac{\kp}{2N} \sum_{k=1}^N a_{ik} \tr(( S_i^\top - \tilde S_i^\top)(S_i-\tilde S_i) ( \tilde S_k^\top \tilde S_i - I_p)) \\
& =: \mathcal I_{131} + \mathcal I_{132} + \mathcal I_{133} .
\end{align*}
Since $\mathcal I_{12}$ and $\mathcal I_{13}$ have a similar structure, if we closely follow \eqref{C-25}, then we have 
\begin{equation} \label{C-28}
\mathcal I_{13} \leq \mathcal I_{131} - \frac{\kp}{2N} \sum_{k=1}^N a_{ik} \|S_i - \tilde S_i\|^2 + \frac{\kp \mathcal Z(t) }{2N} \sum_{k=1}^N |a_{ik} |\|S_i - \tilde S_i\|^2.
\end{equation} 
For the estimate of $\mathcal I_1$ in \eqref{C-18}, we collect the estimates for $\mathcal I_{11}$ in \eqref{C-19}, $\mathcal I_{12}$ in \eqref{C-25} and $\mathcal I_{13}$ in \eqref{C-28} to obtain 
\begin{align} \label{C-30}
\begin{aligned}
 2\mathcal I_1 & \leq \frac\kp N \sum_{k=1}^N |a_{ik}| \|S_i - \tilde S_i\| \|S_k - \tilde S_k \| \\
& \hspace{0.5cm}+  \mathcal I_{121} -\frac{\kp}{2N} \sum_{k=1}^N a_{ik} \|S_i - \tilde S_i\|^2  + \frac{\kp \mathcal Z(t) }{2N}\sum_{k=1}^N |a_{ik}|  \|S_i - \tilde S_i\|^2\\
& \hspace{0.5cm}+ \mathcal I_{131} - \frac{\kp}{2N} \sum_{k=1}^N a_{ik} \|S_i - \tilde S_i\|^2 + \frac{\kp \mathcal Z(t) }{2N} \sum_{k=1}^N |a_{ik} |\|S_i - \tilde S_i\|^2. 
\end{aligned}
\end{align}

Next, we turn to the estimate of $\mathcal I_2$ which is similar to the case of $\mathcal I_1$.  Thus, we omit several details. 

\noindent $\bullet$ (Calculation of $\mathcal I_2$): we observe
\begin{align*}
 2\mathcal I_2 & = \tr (( \dot S_i^\top - \dot{ \tilde S}_i^\top) (S_i - \tilde S_i)) \\
& = \frac{\kp}{N} \sum_{k=1}^N a_{ik} \tr(( S_k^\top - \tilde S_k^\top)(S_i - \tilde S_i)) \\
& \hspace{0.5cm} -\frac{\kp}{2N} \sum_{k=1}^N a_{ik} \tr(( S_k^\top S_i S_i^\top - \tilde S_k^\top \tilde S_i \tilde S_i^\top)(S_i - \tilde S_i)) \\
& \hspace{0.5cm} -\frac{\kp}{2N} \sum_{k=1}^N a_{ik} \tr(( S_i^\top S_k S_i^\top - \tilde S_i^\top \tilde S_k \tilde S_i^\top)(S_i - \tilde S_i))  \\
 & =: \mathcal I_{21}  + \mathcal I_{22} + \mathcal I_{23} .
\end{align*}

$\diamond$ (Estimate of $\mathcal I_{21}$): we find
\[
\mathcal I_{21} \leq \frac\kp N \sum_{k=1}^N |a_{ik}| \|S_i - \tilde S_i \|  \|S_k - \tilde S_k\| .
\]

$\diamond$ (Estimate of $\mathcal I_{22}$): we find
\begin{align*}
& S_k^\top S_i S_i^\top  -  \tilde S_k^\top \tilde S_i \tilde S_i^\top = ( S_k^\top S_i - \tilde S_k^\top \tilde S_i) \tilde S_i^\top + S_k^\top S_i ( S_i^\top - \tilde S_i^\top)    \\
& = ( S_k^\top S_i - \tilde S_k^\top \tilde S_i) \tilde S_i^\top + (S_i^\top - \tilde S_i^\top ) + (S_k^\top S_i - I_p)(S_i ^\top - \tilde S_i^\top),
\end{align*} 
and this gives
\begin{align*}
\mathcal I_{22} & = -\frac{\kp}{2N} \sum_{k=1}^N a_{ik} \tr(( S_k^\top S_i S_i^\top - \tilde S_k^\top \tilde S_i \tilde S_i^\top)(S_i - \tilde S_i)) \\
& = -\frac{\kp}{2N} \sum_{k=1}^N a_{ik} \tr( ( S_k^\top S_i - \tilde S_k^\top \tilde S_i) \tilde S_i^\top ( S_i -\tilde S_i)) -\frac{\kp}{2N} \sum_{k=1}^N a_{ik} \tr(( S_i^\top -\tilde S_i^\top)(S_i - \tilde S_i)) \\
&\hspace{0.5cm}  - \frac{\kp}{2N} \sum_{k=1}^N a_{ik} \tr(( S_k^\top S_i - I_p)(S_i^\top - \tilde S_i^\top)(S_i-\tilde S_i))  \\
& =: \mathcal I_{221} + \mathcal I_{222} + \mathcal I_{223} .
\end{align*}

$\circ$ (Estimates of $\mathcal I_{222}$ and $\mathcal I_{233}$) : 
\begin{align*}
\mathcal I_{222} & = -\frac{\kp}{2N} \sum_{k=1}^N a_{ik} \tr(( S_i^\top -\tilde S_i^\top)(S_i - \tilde S_i)) = -\frac{\kp}{2N} \sum_{k=1}^N a_{ik} \|S_i  - \tilde S_i\|^2 , \\
\mathcal I_{223} & =  -\frac{\kp}{2N} \sum_{k=1}^N a_{ik} \tr(( S_k^\top S_i - I_p)(S_i^\top - \tilde S_i^\top)(S_i-\tilde S_i))  \leq \frac{\kp \mathcal Z(t)}{2N} \sum_{k=1}^N |a_{ik} | \|S_i - \tilde S_i\|^2 .
\end{align*}
Thus, $\mathcal I_{22}$ satisfies
\begin{align*}
\mathcal I_{22} \leq \mathcal I_{221}   -\frac{\kp}{2N} \sum_{k=1}^N a_{ik} \|S_i  - \tilde S_i\|^2 + \frac{\kp \mathcal Z(t)}{2N} \sum_{k=1}^N |a_{ik} | \|S_i - \tilde S_i\|^2.
\end{align*}
Similarly for $\mathcal I_{23}$, we use 
\begin{align*}
&S_i^\top S_k S_i^\top - \tilde S_i^\top \tilde S_k  \tilde S_i^\top =( S_i^\top S_k - \tilde S_i^\top \tilde S_k) \tilde S_i^\top  +  S_i^\top S_k ( S_i^\top  - \tilde S_i^\top)   \\
& = ( S_i^\top S_k - \tilde S_i^\top \tilde S_k) \tilde S_i^\top + (S_i^\top - \tilde S_i^\top) + (S_i^\top S_k - I_p)(S_i^\top - \tilde S_i^\top) 
\end{align*}
to find
\begin{align*}
\mathcal I_{23} &=  -\frac{\kp}{2N} \sum_{k=1}^N a_{ik} \tr(( S_i^\top S_k S_i^\top - \tilde S_i^\top \tilde S_k \tilde S_i^\top)(S_i - \tilde S_i)) \\
& \hspace{0.5cm} -\frac{\kp}{2N} \sum_{k=1}^N a_{ik} \tr(( S_i^\top S_k - \tilde S_i^\top \tilde S_k)\tilde S_i^\top ( S_i - \tilde S_i))  -\frac{\kp}{2N} \sum_{k=1}^N a_{ik} \tr(( S_i^\top - \tilde S_i^\top )(S_i - \tilde S_i)) \\
& \hspace{0.5cm}-\frac{\kp}{2N} \sum_{k=1}^N a_{ik} \tr(( S_i^\top S_k - I_p) (S_i^\top - \tilde S_i^\top )(S_i - \tilde S_i))  \\
& =: \mathcal I_{231} + \mathcal I_{232} + \mathcal I_{233} .
\end{align*}
Hence, $\mathcal I_{23}$ satisfies
\[
\mathcal I_{23} \leq \mathcal I_{231}   -\frac{\kp}{2N} \sum_{k=1}^N a_{ik} \|S_i  - \tilde S_i\|^2 + \frac{\kp \mathcal Z(t)}{2N} \sum_{k=1}^N |a_{ik} | \|S_i - \tilde S_i\|^2.
\] 
Now, we obtain the estimate for $\mathcal I_2$:
\begin{align}  \label{C-40}
\begin{aligned}
 2\mathcal I_2 & \leq \frac\kp N \sum_{k=1}^N |a_{ik}| \|S_i - \tilde S_i \| \|S_k - \tilde S_k\| \\
& + \mathcal I_{221}   -\frac{\kp}{2N} \sum_{k=1}^N a_{ik} \|S_i  - \tilde S_i\|^2 + \frac{\kp \mathcal Z(t)}{2N} \sum_{k=1}^N |a_{ik} | \|S_i - \tilde S_i\|^2 \\
& +\mathcal I_{231}   -\frac{\kp}{2N} \sum_{k=1}^N a_{ik} \|S_i  - \tilde S_i\|^2 + \frac{\kp \mathcal Z(t)}{2N} \sum_{k=1}^N |a_{ik} | \|S_i - \tilde S_i\|^2.
\end{aligned} 
\end{align} 
So far, we have obtained the estimates of $\mathcal I_1$ in \eqref{C-30}   and $\mathcal I_2$ in \eqref{C-40}. To this end, we return to \eqref{C-15}: 
\begin{align} \label{C-42}
\begin{aligned}
\frac12\frac\dd\dt \|S_i - \tilde S_i\|^2 = \mathcal I_1 + \mathcal I_2  &\leq  \frac{1}{2}(\mathcal I_{121} + \mathcal I_{131} + \mathcal I_{221} + \mathcal I_{231} ) \\
&\hspace{0.5cm}+ \frac{\kp}{N} \sum_{k=1}^N |a_{ik}| \|S_i-\tilde S_i\| \|S_k - \tilde S_k\|  -\frac{\kp}{N} \sum_{k=1}^N a_{ik} \|S_i - \tilde S_i\|^2 \\
&\hspace{0.5cm}+ \frac{\kp \mathcal Z(t)}{N} \sum_{k=1}^N |a_{ik}|\|S_i - \tilde S_i\|^2.
\end{aligned}
\end{align} 
For the term $\mathcal I_{121} + \mathcal I_{131} + \mathcal I_{221} + \mathcal I_{231}$, we observe 
\begin{align*}
 &\mathcal I_{121} + \mathcal I_{131} + \mathcal I_{221} + \mathcal I_{231}  \\
&\hspace{0.5cm} = -\frac{\kp}{2N} \sum_{k=1}^N a_{ik} \tr( ( S_i ^\top - \tilde S_i^\top) S_i ( S_i^\top S_k - \tilde S_i^\top \tilde S_k)) -\frac{\kp}{2N} \sum_{k=1}^N a_{ik} \tr(( S_i^\top - \tilde S_i^\top)S_i(S_k^\top S_i - \tilde S_k^\top \tilde S_i))  \\
&\hspace{1cm}  -\frac{\kp}{2N} \sum_{k=1}^N a_{ik} \tr( ( S_k^\top S_i - \tilde S_k^\top \tilde S_i) \tilde S_i^\top ( S_i -\tilde S_i))  -\frac{\kp}{2N} \sum_{k=1}^N a_{ik} \tr(( S_i^\top S_k - \tilde S_i^\top \tilde S_k)\tilde S_i^\top ( S_i - \tilde S_i))  \\
& \hspace{0.5cm}= -\frac{\kp}{2N} \sum_{k=1}^N a_{ik} \tr(( S_i^\top S_k - \tilde S_i^\top \tilde S_k) (( S_i^\top - \tilde S_i^\top) S_i  + \tilde S_i^\top ( S_i - \tilde S_i))) \\
& \hspace{1cm}- \frac{\kp}{2N} \sum_{k=1}^N a_{ik} \tr(( S_k^\top S_i - \tilde S_k^\top \tilde S_i) ( ( S_i^\top - \tilde S_i^\top ) S_i  + \tilde S_i^\top ( S_i - \tilde S_i))) =0. 
\end{align*} 
This completes the proof.

\end{document}